\documentclass[11pt,reqno,a4paper]{amsart}
\usepackage[utf8]{inputenc}
\usepackage[T1]{fontenc}
\usepackage{graphicx}
\usepackage{longtable}
\usepackage{wrapfig}
\usepackage{rotating}
\usepackage[normalem]{ulem}
\usepackage{amsmath}
\usepackage{amssymb}
\usepackage{capt-of}
\usepackage{hyperref}
\author{Thomas Hüttemann}\address[T.~Hüttemann]{Queen's University Belfast, School of Mathematics and Physics, Mathematical Sciences Research Centre, Belfast BT7~1NN, UK}\email{t.huettemann@qub.ac.uk}
\author{Dan Kucerovsky}\address[D.~Kucerovsky]{University of New Brunswick, Department of Mathematics and Statistics, P.O.~Box 4400, Fredericton, NB, Canada, E3B~5A3}\email{dkucerov@unb.ca}\thanks{Work on this paper started during a research visit of the second-named author to Queen's University Belfast. Financial support of the Fredrik and Catherine Eaton Fund is gratefully acknowledged}
\usepackage{amsthm}
\usepackage{amssymb}
\usepackage{tikz-cd}
\usepackage{mathtools}
\newtheorem{lem}{Lemma}
\newtheorem{prop}[lem]{Proposition}
\newtheorem{thm}[lem]{Theorem}

\theoremstyle{definition}
\newtheorem{dfn}[lem]{Definition}
\newtheorem{xmpl}[lem]{Example}

\numberwithin{lem}{section}
\numberwithin{equation}{lem}
\newcommand{\bN}{\mathbb{N}}
\newcommand{\bZ}{\mathbb{Z}}

\newcommand{\bR}{\mathbb{R}}

\newcommand{\tensor}{\otimes}
\newcommand{\rcof}[1][]{\arrow[r, tail,   "#1"']}
\newcommand{\rcofp}[1][]{\arrow[r, tail,   "#1"]}

\newcommand{\po}{\ar[dr,phantom,"{\tikz{\draw[-] (-0.3,0) -- (-0.3,0.3) -- (0,0.3); \path (-.15,.15) node[circle,fill=black,scale=0.2] {};}}"]}

\usepackage{enumerate}
\date{\today}
\title{The Abel summation method and the infinite Euler characteristic}
\hypersetup{
 pdfauthor={Thomas Hüttemann and Dan Kucerovsky},
 pdftitle={The Abel summation method and the infinite Euler characteristic},
 pdfkeywords={},
 pdfsubject={},
 pdfcreator={Emacs 29.3 (Org mode 9.6.15)}, 
 pdflang={English}}
\begin{document}

\begin{abstract}
We develop a finiteness notion for unbounded chain complexes over a
commutative noetherian integral domain~\(R\) employing the Abel
summation method. The algebraic \(K\)-theory of such complexes is
defined, and shown to be non-trivial. We also exhibit a natural map
from the (usual) algebraic \(K\)-theory of~\(R\) into the new
\(K\)-theory and show that its image contains a canonical infinite
cyclic subgroup.
\end{abstract}

\maketitle
\newcommand{\Alim}{\text{\rm A-}\lim}
\newcommand{\adm}{A-admissible}
\newcommand{\ca} {\chi_{\text{A}}}
\newcommand{\ci} {\chi_{\infty}}
\newcommand{\bc}{\boldsymbol{\bar\chi}}
\newcommand{\CA}{\mathcal{C}_{\text{\upshape A}}(R)}
\newcommand{\CAf}{\CA_{\text{\upshape f}}}
\newcommand{\fP}{\mathfrak{P}}

\section{Introduction}
\label{sec:orgcbc80c5}
The Euler characteristic of a bounded chain complex~\(C\) is defined
as the alternating sum of the ranks of the homology modules of~\(C\):
\begin{equation*}
  \chi(C) = \sum_{n} (-1)^{n} \mathrm{rank}\, H_{n} C
\end{equation*}
The definition does not generalise easily to unbounded chain
complexes. One possibility is to consider the Euler characteristic of
finite sections of~\(C\), weighted inversely proportional to their
length, and pass to the limit as the length is increased. Developing
this idea leads to a definition of an Euler characteristic for
unbounded chain complexes employing the Hölder summation method of
sequences \cite{HK1}.

In the present paper, we demonstrate that the (stronger) Abel
summation method can be used instead. This requires a careful
adjustment of the arguments used in the previous paper. Crucially, to
obtain a category of chain complexes with useful notions of
cofibrations and weak equivalences (allowing the definition of higher
algebraic \(K\)-groups), we need to impose growth conditions on the
ranks of the homology modules which differs noticeably from the
previously used one (Definition~\ref{org3cfd71c} and
\cite[Definition~2.4]{HK1}). The resulting category contains the one
defined in \cite{HK1} using the Hölder summation method, has
uncountable Grothendieck group, and contains a canonical infinite
cyclic subgroup generated by the image of \([R] \in K_{0}(R)\).

\section{Conventions}
\label{sec:org0a069a7}
We denote by~\(\bN\) the set of non-negative integers (so \(0 \in
\bN\)). Throughout the paper, \(R\) stands for an arbitrary
commutative noetherian integral domain with field of
quotients~\(Q\). The rank of an \(R\)-module \(M\) is defined as
\begin{equation*}
  \mathrm{rank}\, M = \dim_{Q} M \tensor_{R} Q \ .
\end{equation*}
By a chain complex \(C\) we mean a sequence of \(R\)-modules
and \(R\)-linear maps
\begin{equation*}
  C \colon
  \ldots \leftarrow C_{n-1}
  \leftarrow C_{n} \leftarrow C_{n+1}
  \leftarrow \ldots
\end{equation*}
such that the composition of any two successive maps in the complex
is~0. All chain complexes in this paper are assumed to be \emph{positive}
in the sense that \(C_{n} = 0\) whenever \(n<0\).

\section{Abel limits and Abel sequences}
\label{sec:org598ca50}
Given a sequence \(\mathbf{a} = (a_{n})_{n \in \bN}\) of real numbers
we define the power series
\begin{equation*}
  \bc \mathbf{a} = \sum_{k=0}^\infty a_{k} x^k \ .
\end{equation*}
We let \(R(\mathbf{a})\) denote the radius of convergence of this
power series, and call this the radius of convergence of the
sequence~\(\mathbf{a}\). 

We are exclusively concerned with the case \(R(\mathbf{a}) \geq
1\). By the usual Cauchy-Hadamard formula, this condition is
equivalent to saying
\begin{equation*}
  \lim \sup_{n \geq 1} \sqrt[n]{|a_{n}|} \leq 1
\end{equation*}
(see \cite[§24, p.76]{zbMATH03003164}); by the ratio test for
convergence, it is also equivalent to the requirement
\begin{equation*}
  \lim_{n \in \bN} \left| \frac{a_{n}}{a_{n+1}} \right| \geq 1 \ ,
\end{equation*}
provided this limit exists.

\begin{lem}
\label{org21dd538}
\begin{enumerate}[{\rm (a)}]
\item \label{org6863a36} If \(R(\mathbf{a}) \geq 1\) and \(R(\mathbf{b})
   \geq 1\), then also \(R(\mathbf{a} \pm \mathbf{b}) \geq 1\) and
\(R(c \cdot \mathbf{a}) \geq 1\) for any constant \(c \in \bR\).
\item \label{org6afd178} If \(R(\mathbf{b}) \geq 1\) and \(|\mathbf{a}|
   \leq |\mathbf{b}|\) , then \(R(\mathbf{a}) \geq 1\).
\end{enumerate}
\end{lem}

\begin{proof}
Part~(\ref{org6863a36}) is clear. Part~(\ref{org6afd178}) holds since a
power series converges absolutely on its (open) interval of
convergence.
\end{proof}

\begin{dfn}
We define the \emph{Abel limit}, or \emph{A-limit}, of the sequence \(\mathbf{a}
= (a_{n})_{n \in \bN}\) by the expression
\begin{equation*}
  \Alim \mathbf{a} = \lim_{x \to 1^{-}} (1-x) \bc \mathbf{a} \ ,
\end{equation*}
provided this limit exists (as a finite real number); for this to make
sense, we implicitly assume \(R(\mathbf{a}) \geq 1\). If \(\Alim
\mathbf{a}\) is defined we say that \(\mathbf{a}\) is an \emph{Abel
sequence}. If an Abel sequence has Abel limit~0, the sequence is
called an \emph{Abel null} sequence.
\end{dfn}

\begin{xmpl}
\label{org2ff51a2} Suppose that the sequence \(\mathbf{a}\) alternates
between the values \(s\) and~\(t\) so that \(a_{2n} = s\) and
\(a_{2n+1} = t\) for all~\(n\). Then
\begin{equation*}
  \bc(\mathbf{a}) = s + tx + sx^2 + tx^3 + \ldots
  = \frac{s+tx}{1-x^2}
\end{equation*}
so that the Abel limit
\begin{equation*}
  \Alim \mathbf{a} = \lim_{x \to 1^{-}} (1-x) \cdot \frac{s+tx}{1-x^2}
  = \lim_{x \to 1^{-}} \frac{s+tx}{1+x} = \frac{s+t}{2}
\end{equation*}
is the arithmetic mean of \(s\) and~\(t\).
\end{xmpl}

\begin{lem}
\label{org44904eb} Let \(\mathbf{a}\) be a sequences possessing an
Abel limit, and let \(\mathbf{b}\) be obtained from~\(\mathbf{a}\) by
inserting, deleting or modifying finitely many terms. Then
\(\mathbf{b}\) also possesses an Abel limit, and \(\Alim \mathbf{a} =
\Alim \mathbf{b}\).
\end{lem}

\begin{proof}
It is enough to discuss a special case:
\begin{quote}
  \textit{Suppose that \(\mathbf{b}\) is obtained from \(\mathbf{a}\) by
  deleting the first term such that \(b_n = a_{n+1}\). Then
  \(\Alim\mathbf{a}\) exists if and only if \(\Alim \mathbf{b}\) exists,
  in which case the limits agree.}
\end{quote}
(In the terminology of summation theory, this says that Abel limits
are a \emph{translative} or \emph{regular} summation method
\cite[p.115]{zbMATH03060297}.) But this follows readily from the power
series identity \(a_0 + x \cdot \bc \mathbf{b} = \bc \mathbf{a}\).
\end{proof}

\begin{lem}
\label{org8f22aeb}
\begin{enumerate}[{\rm (a)}]
\item \label{org62d69ce} The Abel limit is linear such that \(\Alim
   (\mathbf{a} + c \cdot \mathbf{b}) = \Alim\mathbf{a} + c \cdot \Alim
   \mathbf{b}\) for all Abel sequences \(\mathbf{a}\)
   and~\(\mathbf{b}\), and all \(c \in \bR\).
\item \label{org9c69999} If \(|\mathbf{a}| \leq \mathbf{b}\) and
\(\mathbf{b}\) is Abel null
then likewise \(\mathbf{a}\) is Abel null.
\end{enumerate}
\end{lem}

\begin{proof}
Part~(\ref{org62d69ce}) is clear. Part~(\ref{org9c69999}): We know
\(R(\mathbf{a}) \geq 1\) by Lemma~\ref{org21dd538}~(\ref{org6afd178}). The
inequalities \(|a_i| \leq b_i\) imply \(\bc |\mathbf{a}| \leq \bc
\mathbf{b}\) for \(x \in (0,1)\). Thus
\begin{equation*}
  0 \leq \left| (1-x) \bc \mathbf{a} \right|
  \leq (1-x) \bc |\mathbf{a}|
  \leq (1-x) \bc \mathbf{b}
  \quad \to 0 \qquad \text{as \(x \to 1^{-}\)}
\end{equation*}
whence \(\Alim \mathbf{a} = 0\).
\end{proof}

We will also have occasion to consider sequences which may not have an
Abel limit but are such that the function represented by its
associated power series do not grow too fast when \(x\) approaches~1
along the interval \((0,1)\).

\begin{dfn}
\label{org3cfd71c} We say that a sequence
\(\mathbf{a}\) with \(R(\mathbf{a}) \geq 1\) is \emph{of controlled growth}
if its associated power series growths sub-quadratically as \(x \to
1^{-}\); in symbols, we require \((1-x)^2 \cdot \bc \mathbf{a} \to 0\)
as \(x \to 1^{-}\). We say that \(\mathbf{a}\) is \emph{of controlled
absolute growth} if \(|\mathbf{a}|\) is of controlled growth, that is,
if \((1-x)^2 \cdot \bc |\mathbf{a}| \to 0\) as \(x \to 1^{-}\).
\end{dfn}

Note that if \(|\mathbf{a}|\) possess an Abel limit, then
\(\mathbf{a}\) is of controlled absolute growth.

\begin{xmpl}
The sequence \(\mathbf{a}\) with \(a_n = (-1)^{n} (n+1)(n+2)\) is of
controlled growth. Indeed,
\begin{equation*}
  \bc \mathbf{a} = 2 - 6x + 12x^2 - 20x^3 + \ldots = \frac{2}{(1+x)^3}
\end{equation*}
so that \((1-x)^2 \cdot \bc \mathbf{a} \to 0\) as \(x \to 1^{-}\). The
sequence \(\mathbf{a}\) is, however, \emph{not} of controlled absolute
growth since
\begin{equation*}
  \bc |\mathbf{a}| = 2 + 6x + 14x^2 + 20x^3 + \ldots = \frac{2}{(1-x)^3}
\end{equation*}
whence \((1-x)^2 \cdot \bc |\mathbf{a}| \to \infty\) as \(x \to
1^{-}\).
\end{xmpl}

The following statements will be of use later on:

\begin{lem}
\label{orgbed5e77} Let \(\mathbf{a}\) and~\(\mathbf{b}\) be sequences with
\(R(\mathbf{a}) \geq 1\) and \(R(\mathbf{b}) \geq 1\).
\begin{enumerate}[{\rm (a)}]
\item \label{orgc8423e5} If \(0 \leq \mathbf{a} \leq \mathbf{b}\) and
\begin{equation*}
  \lim_{x \to 1^{-}} (1-x)^2 \cdot \bc \mathbf{b} = 0
\end{equation*}
then also
\begin{equation*}
  \lim_{x \to 1^{-}} (1-x)^2 \cdot \bc \mathbf{a} = 0 \ .
\end{equation*}
\item \label{orga0c1b40} If \((1-x)^2 \cdot \bc |\mathbf{a}| \to 0\)
as \(x \to 1^{-}\), then \((1-x)^2 \cdot \bc \mathbf{a} \to 0\) as
well. In other words, controlled absolute growth implies controlled
growth.
\end{enumerate}
\qed
\end{lem}

\section{The Abel method version of the infinite Euler characteristic}
\label{sec:orgb777705}
\label{org29e20d3}
For a positive chain complex~\(C\) of \(R\)-modules, we define the
sequence \(\mathbf{H}C\) by setting
\begin{equation*}
  (\mathbf{H}C)_{n} = (-1)^n \mathrm{rank}\, H_{n}C \ ,
\end{equation*}
and the power series \(\bc(C)\) by the assignment
\begin{equation*}
  \bc(C) = \bc (\mathbf{H}C) \ .
\end{equation*}
Its radius of convergences is called the radius of convergence
of~\(C\) and will be denoted \(R(C)\). Thus \(R(C) = R(\mathbf{H}C)\)
by definition.

\begin{dfn}
We call \(C\) an \emph{Abel complex} if \(\mathbf{H}C\) is an Abel
sequence. We say that \(C\) is \emph{of controlled absolute growth} if the
sequence \(\mathbf{H}C\) is of controlled absolute growth.

The category with objects all Abel complexes of controlled absolute
growth and morphisms all chain maps between them is denoted~\(\CA\).
\end{dfn}

\begin{dfn}
Let \(C\) be an Abel complex. The \emph{Abel method infinite Euler
characteristic} of~\(C\), or \emph{AE characteristic} of~\(C\) for short,
is defined as
\begin{equation*}
  \ca C = \Alim \mathbf{H}C \ .
\end{equation*}
\end{dfn}

The next Proposition provides some justification for the name "Euler
characteristic". Informally, the AE characteristic generalises the
idea of computing the Euler characteristic of an unbounded complex by
computing the (classical) Euler characteristic of successively longer
segments of the complex, weighted inversely proportional to the
length.

\begin{prop}
\label{org09c196b} If \(C\) is a chain complex such that its Hölder
method infinite Euler characteristic \(\chi_{H} C\) of \cite{HK1} is
defined, then \(C\) is an Abel complex and the equality \(\ca C =
\chi_H C\) holds.
\end{prop}

\begin{proof}
This is true since the Abel summation method is stronger than and
consistent with the Hölder summation method, see \cite[§55.II, p111]{zbMATH03139022} and the references given there.
\end{proof}

If \(C\) is a bounded complex of finitely generated \(R\)-modules so
that its classical Euler characteristic \(\chi(C)\) is defined, then
\(C\) has only finitely many non-vanishing homology modules so that
\(\ca C = 0\). Thus the AE characteristic is of interest only where
the classical Euler characteristic cannot be employed.

\begin{xmpl}
Let \(C\) be the singular chain complex with coefficients in~\(R\) of
the infinite one-point union \(\bigvee_{n \geq 1} S^{n}\) of
positive-dimensional spheres. Then \((\mathbf{H}C)_{n} = (-1)^n\) for
all \(n \in \bN\). Consequently, \(\bc(C) = \sum_{n \in \bN} (-1)^n
x^{n} = 1/(1+x)\) and \(\bc |\mathbf{H}C| = 1/(1-x)\) for \(|x|<1\) so
that
\begin{gather*}
  \ca(C) = \lim_{x \to 1^{-}} (1-x) \bc(C) =
  \lim_{x \to 1^{-}} (1-x) \cdot \frac{1}{1+x} = 0
  \intertext{and}
  (1-x)^2 \cdot \bc |\mathbf{H}C| = \frac{(1-x)^2}{1-x} = 1-x
  \to 0 \text{ as } x \to 1^{-} \ .
\end{gather*}
In particular, \(C\) is an Abel complex of controlled absolute growth.
\end{xmpl}

\begin{xmpl}
Let \(C\) be the singular chain complex with coefficients in~\(R\) of
the infinite one-point union \(\bigvee_{n \geq 1} S^{2n}\) of
even-dimensional spheres. Then \((\mathbf{H}C)_{2n} = 1\) and
\((\mathbf{H}C)_{2n+1} = 0\) for all \(n \in \bN\). Consequently,
\(\bc(C) = \sum_{n \in \bN} x^{2n} = 1/(1-x^2) = \bc |\mathbf{H}C|\)
for \(|x|<1\) so that
\begin{gather*}
  \ca(C) = \lim_{x \to 1^{-}} (1-x) \bc(C) =
  \lim_{x \to 1^{-}} (1-x) \cdot \frac{1}{1-x^2} =
  \lim_{x \to 1^{-}} \frac{1}{1+x} = \frac12
  \intertext{and}
  (1-x)^2 \cdot \bc |\mathbf{H}C| = \frac{(1-x)^2}{1-x^2} = \frac{1-x}{1+x}
  \to 0 \text{ as } x \to 1^{-} \ .
\end{gather*}
In particular, \(C\) is an Abel complex of controlled absolute growth.
\end{xmpl}

\begin{xmpl}
Let \(C\) be the singular chain complex with coefficients in~\(R\) of
an infinite one-point union \(\bigvee_{n \geq 1} \bigvee_{k=0}^{n}
S^{n}\) of \(n+1\) copies of the \(n\)-dimensional sphere, for all \(n
\geq 1\). Then \((\mathbf{H}C)_{n} = (-1)^{n} (n+1)\) for all \(n \in
\bN\). Consequently, \(\bc(C) = \sum_{n \in \bN} (-1)^n (n+1)x^{n} =
1/(1+x)^2\) and \(\bc |\mathbf{H}C| = 1/(1-x)^2\) for \(|x|<1\) so
that
\begin{gather*}
  \ca(C) = \lim_{x \to 1^{-}} (1-x) \bc(C) =
  \lim_{x \to 1^{-}} (1-x) \cdot \frac{1}{(1+x)^2} = 0
  \intertext{and}
  (1-x)^2 \cdot \bc |\mathbf{H}C| = \frac{(1-x)^2}{(1-x)^2} = 1 \not\to 0
  \text{ as } x \to 1^{-} \ .
\end{gather*}
In particular, \(C\) is an Abel complex but \emph{not} of controlled
absolute growth.
\end{xmpl}

\section{Admissible short exact sequences}
\label{sec:org49684a9}
A short exact sequence \(\mathcal{S} \colon 0 \to A \to B \to C \to
0\) of positive chain complexes gives rise to a long exact sequence of
homology modules with connecting homomorphisms \(\delta_{n+1} \colon
H_{n+1}C \to H_{n} A\), and thence to exact sequences
\begin{equation*}
  0 \to \mathrm{im}\, \delta_{n+1} \to H_{n} A \to H_{n} B \to H_{n} C
  \to \mathrm{im}\, \delta_{n} \to 0
\end{equation*}
which in turn yield the relations
\begin{equation*}
  \mathrm{rank}\, \mathrm{im}\, \delta_{n+1} +
  \mathrm{rank}\, \mathrm{im}\, \delta_{n} =
  \mathrm{rank}\, H_{n} A -
  \mathrm{rank}\, H_{n} B +
  \mathrm{rank}\, H_{n} C \ .
\end{equation*}
Multiplying by \((-1)^n x^n\) and adding up establishes the power
series identity
\begin{equation}
\label{org338421f}
 (1-x) \cdot \bc (\boldsymbol{\delta} \mathcal{S}) =
  \bc (A) - \bc (B) + \bc (C)
\end{equation}
where we have used the sequence \(\boldsymbol{\delta} \mathcal{S} =
\big( (-1)^n \cdot \mathrm{rank}\, \mathrm{im}\, \delta_{n+1})_{n \in
\bN}\) associated with the short exact sequence~\(\mathcal{S}\).

\begin{dfn}
The short exact sequence of chain complexes~\(\mathcal{S}\) is called
\emph{weakly \adm} if the sequence~\(\boldsymbol{\delta} \mathcal{S}\) is
of controlled growth, that is, if it satisfies
\begin{equation*}
  \lim_{x \to 1^{-}} (1-x)^2 \cdot
  \bc(\boldsymbol{\delta} \mathcal{S}) \to 0 \ ,
\end{equation*}
and is called \emph{\adm{}} if it is of controlled absolute growth, that
is, if it satisfies
\begin{equation*}
  \lim_{x \to 1^{-}} (1-x)^2 \cdot
  \bc \big( |\boldsymbol{\delta} \mathcal{S}| \big) \to 0 \ .
\end{equation*}
\end{dfn}

We record that all split short exact sequences \(0 \to A \to A \oplus
C \to C \to 0\) are \adm{} since the connecting
homomorphisms~\(\delta_n\) vanish identically whence
\(\bc(\boldsymbol{\delta} \mathcal{S}) = 0\). Every \adm{} short exact
sequence is also weakly \adm{} in view of
Lemma~\ref{orgbed5e77}(\ref{orga0c1b40}).

Admissibility is a void condition when restricted to short exact
sequences of chain complexes of controlled absolute growth:

\begin{lem}
\label{org3c584b5} Suppose that at least one of \(\mathbf{H}A\)
and~\(\mathbf{H}C\) is of controlled absolute growth. Then the short
exact sequence \(\mathcal{S}\) is \adm{}.
\end{lem}

\begin{proof}
Suppose that \(\mathbf{H}C\) is of controlled absolute growth, then so
will be the sequence~\(\mathbf{a}=(a_{n})\) with \(a_n =
\mathrm{rank}\, H_{n+1} C = |\mathbf{H}C|_{n+1}\). As
\(|\boldsymbol{\delta} \mathcal{S}|_{n} = \mathrm{rank}\,
\mathrm{im}\, (\delta_{n+1} \colon H_{n+1} C \to H_{n} A)\), and as
homomorphisms cannot increase the rank, we know that
\(|\boldsymbol{\delta} \mathcal{S}| \leq \mathbf{a}\). Hence by
Lemma~\ref{orgbed5e77}~(\ref{orgc8423e5}), \(\boldsymbol{\delta}
\mathcal{S}\) is of controlled absolute growth.

If \(\mathbf{H}A\) is of controlled absolute growth, we argue
similarly that \(\mathrm{im}\, \delta_{n+1} = \ker (f \colon H_{n} A
\to H_{n} B)\) by exactness so that \(|\boldsymbol{\delta}
\mathcal{S}| \leq |\mathbf{H}A|\). By
Lemma~\ref{orgbed5e77}~(\ref{orgc8423e5}), \(\boldsymbol{\delta}
\mathcal{S}\) is of controlled absolute growth.
\end{proof}

The Lemma implies that any short exact sequence of objects of~\(\CA\)
is automatically \adm{}.

\begin{prop}
\label{orgfab7294} The AE characteristic is additive for weakly \adm{}
short exact sequences of Abel complexes. That is, if \(0 \to A \to B
\to C \to 0\) is a weakly \adm{} short exact sequence then \(\ca(A) -
\ca(B) + \ca(C) = 0\).
\end{prop}

\begin{proof}
Multiplying relation~(\ref{org338421f}) by \((1-x)\) yields
\begin{equation*}
  (1-x)^2 \cdot \bc (\boldsymbol{\delta} \mathcal{S}) =
  (1-x) \cdot \bc (A) - (1-x) \cdot \bc (B) + (1-x) \cdot \bc (C)
  \ ,
\end{equation*}
and in the limit \(x \to 1^{-}\) the left-hand side vanishes by the
very definition of weakly \adm{}, while the right-hand side turns into
the alternating sum \(\ca(A) - \ca(B) + \ca(C)\).
\end{proof}

For later use, we record the following facts:

\begin{lem}
\label{orgf31cd94} Let \(\mathcal{S} \colon 0 \to A \to B \to C \to
0\) be a short exact sequence of positive chain complexes of
\(R\)-modules.
\begin{enumerate}[{\rm (a)}]
\item \label{org492a86b} If two of the three complexes \(A\), \(B\)
and~\(C\) have radius of convergence at least~1, then so does the
third.
\item \label{org7ce804f} Suppose that the short exact sequence
\(\mathcal{S}\) is weakly \adm{}. If two of the three complexes
\(A\), \(B\) and~\(C\) are Abel complexes then so is the third.
\item \label{org97b90a4} Suppose that two of the three complexes
\(A\), \(B\) and~\(C\) are of controlled absolute growth. Then so
is the third, and \(\mathcal{S}\) is \adm{}.
\item \label{org58dd786} Suppose that two of the three chain
complexes \(A\), \(B\) and~\(C\) are objects of~\(\CA\). Then so is
the third, and \(\mathcal{S}\) is \adm{}.
\end{enumerate}
\end{lem}

\begin{proof}
(\ref{org492a86b}) At least one of the complexes \(A\) and~\(C\) has
radius of convergence \(\geq 1\), and since \(\mathrm{im}\, \delta_{n}
= \ker f_{n-1}\) we have
\begin{equation*}
  \relax |\boldsymbol{\delta} \mathcal{S}| \leq |\mathbf{H}A|
  \quad \text{and} \quad
  \relax |\boldsymbol{\delta} \mathcal{S}| \leq \big| \mathbf{H}C[-1] \big|
\end{equation*}
(compare proof of Lemma~\ref{org3c584b5}) which means, in view of
Lemma~\ref{org21dd538}(\ref{org6afd178}), that \(R(\boldsymbol{\delta}
\mathcal{S}) \geq 1\). Thus the power series \(\bc
(\boldsymbol{\delta} \mathcal{S})\) has radius of convergence at
least~1. Since three of the four power series in equation~(\ref{org338421f})
have radius of convergence at least~1, the same is true for the
fourth.

(\ref{org7ce804f}) By the previous part, all complexes
in~\(\mathcal{S}\) have radius of convergence at least~1. Moreover,
from~(\ref{org338421f}) we obtain
\begin{equation*}
  (1-x)^2 \cdot \bc (\boldsymbol{\delta} \mathcal{S}) =
  (1-x) \cdot \bc (A) - (1-x) \cdot \bc (B) + (1-x) \cdot \bc (C) \ .
\end{equation*}
By hypotheses, the left-hand side vanishes in the limit \(x \to
1^{-}\), while two of the three summands on the right result in
well-defined finite real numbers. Hence the limit exists for the third
summand as well.

(\ref{org97b90a4}) By the first part, all three complexes have
radius of convergence at least one. Multiplying the equation from the
proof of the previous part by \((1-x)\) gives
\begin{equation*}
  (1-x) \cdot (1-x)^2 \cdot \bc (\boldsymbol{\delta} \mathcal{S}) =
  (1-x)^2 \cdot \bc (A) - (1-x)^2 \cdot \bc (B) + (1-x)^2 \cdot \bc (C)
  \ .
\end{equation*}
We already know that \(\mathcal{S}\) is \adm{} by
Lemma~\ref{org3c584b5}, so as \(x \to 1^{-}\) the left-hand side will
vanish. Two of the three terms on the right-hand side will vanish in
the limit as well, hence so does the third, proving that all three
complexes are of controlled absolute growth.

(\ref{org58dd786}) This follows immediately from parts
(\ref{org97b90a4}) and~(\ref{org7ce804f}).
\end{proof}

\section{Algebraic \(K\)-theory}
\label{sec:org0c65a60}
\label{org75802ce}

We equip the category~\(\CA{}\) of Abel complexes of controlled
absolute growth (\S\ref{org29e20d3}) with the structure of a category with
cofibrations and weak equivalences in the sense of Waldhausen
\cite[§1.2]{zbMATH03927168}, as follows:

\begin{itemize}
\item A \emph{cofibration} is an injective chain map \(f \colon C \to D\)
in~\(\CA{}\). The subcategory of \(\CA{}\) consisting of
all objects of~\(\CA{}\) and all cofibrations will be denoted
\(\mathrm{co}\CA{}\). As usual, cofibrations will be written
with the symbol "\(\tikzcd {} \rcof & {} \endtikzcd\)".
\item A \emph{weak equivalence} is a chain map in~\(\CA{}\) that is a
quasi-iso\-mor\-phism (that is, a chain map inducing isomorphisms on
all homology modules). The subcategory of~\(\CA{}\) consisting of all
objects of~\(\CA{}\) and all weak equivalences will be denoted
\(\mathrm{w}\CA{}\).
\end{itemize}

\begin{thm}
\begin{enumerate}[{\rm (a)}]
\item \label{org38c5811} With the structure above, \(\CA{}\) is a
category with cofibrations and weak equivalences in the sense of
Waldhausen \cite{zbMATH03927168} satisfying the
saturation axiom and the extension axiom.
\item \label{org350ab27} The usual mapping cylinder construction for chain
complexes provides a cylinder functor satisfying the cylinder
axiom.
\end{enumerate}
\end{thm}

\begin{proof}
(\ref{org38c5811}) Axioms Cof~1 (isomorphisms are cofibrations)
and Cof~2 (all objects are cofibrant) are trivial. Axiom Cof~3
requires that for each cofibration \(f \colon \tikzcd A \rcof & B
\endtikzcd\) and each map \(A \to C\) in~\(\CA{}\), the pushout
\begin{equation*}
  \begin{tikzcd}
    A \rcofp[f] \dar{} \po & B \dar{} \\
    C \rar{F} & D
  \end{tikzcd}
\end{equation*}
exists in~\(\CA{}\), and that \(F\) is a cofibration.

By taking the pushout in the category of non-negative chain complexes,
we obtain the following diagram of short exact sequences:
\begin{equation*}
  \begin{tikzcd}
    0 \rar{} & A \rar{f} \dar{} \po
      & B \rar{} \dar{}
      & \mathrm{coker}\,f \rar{} \dar{\cong} & 0 \\
    0 \rar{} & C \rar{F} & D \rar{} & \mathrm{coker}\,F \rar{} & 0
  \end{tikzcd}
\end{equation*}
As \(A\) and \(B\) are objects of~\(\CA\) so is \(\mathrm{coker}\,f\),
by Lemma~\ref{orgf31cd94}~(\ref{org58dd786}). Hence the isomorphic
complex \(\mathrm{coker}\,F\) is an object of~\(\CA\). Since \(C\) is
an object of~\(\CA\) by hypothesis,
Lemma~\ref{orgf31cd94}~(\ref{org58dd786}) asserts that \(D = A \cup_B
C\) is an object of~\(\CA\) as well.

Axiom Weq~1 (all isomorphisms are weak equivalences) is trivial, while
axiom Weq~2 (the gluing lemma) is known to hold in the category of
chain complexes, hence holds in particular in the present case. It is
well-known that the weak equivalences (the quasi-isomorphisms) satisfy
the saturation and extension axioms (the latter being a consequence of
the five lemma, applied to long exact homology sequences).

\medbreak

(\ref{org350ab27}) The mapping cylinder \(T(f)\) of a map \(f \colon C \to D\)
of chain complexes is chain homotopy equivalent to~\(D\). Thus
\(\mathbf{H}T(f) = \mathbf{H}D\), and \(T(f)\) is an object of~\(\CA\)
if (and only if) \(D\) is. Thus \(T(f)\) furnishes \(\CA{}\) with a
cylinder functor satisfying the cylinder axiom since the construction
is well-known to provide such a cylinder functor on the category of
all positive chain complexes.
\end{proof}

\begin{dfn}
The \emph{Abel method \(K\)-theory of~\(R\)} is defined to be the Waldhausen
\(K\)-theory space of the category~\(\CA\),
\begin{equation*}
  K^{\text{A}}(R) = \Omega |w \mathcal{S}_{\bullet} \CA| \ ,
\end{equation*}
where the letter "\(w\)" denotes the category of weak equivalences as
usual. The \emph{Abel method \(K\)-groups of~\(R\)} are defined as the
homotopy groups of the \(K\)-theory space,
\begin{equation*}
  K^{\text{A}}_{q} (R) = \pi_{q} K^{\text{A}}(R) \ .
\end{equation*}
\end{dfn}

The group \(K^{\text{A}}_0(R)\) is the abelian group generated by
symbols \([A]\), for each object \(A \in \CA{}\), subject to the
relations
\begin{itemize}
\item \([A] = [A']\) whenever there exists a weak equivalence \(A \to A'\),
\item \([B] = [A] + [B/A]\) for each cofibration \(f \colon \tikzcd A \rcof & B
  \endtikzcd\), where \(B/A = \mathrm{coker}\,(f)\).
\end{itemize}

\begin{thm}
The group \(K^{\text{A}}_{0}(R)\) is uncountable; in particular, the
Abel method \(K\)-theory of~\(R\) is non-trivial.
\end{thm}

\begin{proof}
As all short exact sequences in~\(\CA\) are admissible by
Lemma~\ref{orgf31cd94}(\ref{org58dd786}), and as the AE
characteristic \(\ca\) is additive by Proposition~\ref{orgfab7294}, it
induces a well-defined group homomorphism \(K^{\text{A}}_{0}(R) \to
\bR\). From the proof of \cite[Theorem 2.12]{HK1} we know that, for
any \(r \in \bR\), there exists a positive chain complex~\(C\) having
the following properties:

\begin{itemize}
\item the "Hölder limit" (that is, the limit of the sequence of arithmetic
means) of \(\mathbf{H}C\) equals~\(r\).
\item The complex \(C\) and the sequence \(\mathbf{H} C\) are concentrated
in even degrees if \(r \geq 0\), and in odd degrees if \(r < 0\).
\end{itemize}

\noindent This complex~\(C\) is an object
of~\(\CA\). Indeed, by Proposition~\ref{org09c196b}, the first property
ensures that \(C\) is an Abel complex with \(\ca(C) = r\). The second
property implies that the Hölder limit of \(|\mathbf{H}C|\) equals the
Hölder limit of~\(\mathbf{H}C\) or its negative (\emph{viz.},
equals~\(|r|\)). Consequently, \(|\mathbf{H}C|\) is an Abel sequence
with Abel limit~\(|r|\) whence \(C\) is of controlled absolute growth.

We conclude that \(\ca \colon K^{\text{A}}_{0}(R) \to \bR\) is
surjective, so its domain must be an uncountable group.
\end{proof}

\section{Finite chain modules}
\label{sec:orgaa8680a}
Let \(\CAf\) denote the full subcategory of~\(\CA\) with objects the
Abel complexes of controlled absolute growth consisting of finitely
generated \(R\)-modules.

\begin{thm}
\label{orgbf777be}
\begin{enumerate}
\item The category \(\CAf\) is a subcategory with
cofibrations and weak equivalences of the category~\(\CA\),
and the cylinder functor on~\(\CA\) restricts to a cylinder
functor for~\(\CAf\).
\item The inclusion functor \(\iota \colon \CAf \xrightarrow{\subset}
     \CA\) induces a homotopy equivalence
\begin{equation*}
  {}|w \mathcal{S}_{\bullet} \CAf |
  \xrightarrow{\sim}
  {}|w \mathcal{S}_{\bullet} \CA |
\end{equation*}
and hence an isomorphism on \(K\)-groups.
\end{enumerate}
\end{thm}

\begin{proof}
The first part is immediate. The second part follows from the
approximation theorem; the proof of~\cite[Theorem~7.1]{HK1} applies
verbatim to verify the crucial factorisation hypothesis.
\end{proof}

Given a chain complex \(C\) of modules of finite rank, we define two
further sequences:
\begin{equation*}
  \mathbf{R}C = \big( (-1)^n \mathrm{rank}\, C_{n} \big)_{n \in \bN}
  \qquad \text{and} \qquad
  \mathbf{B}C = \big( (-1)^n \mathrm{rank}\, B_{n}C\big)_{n \in \bN} \ ,
\end{equation*}
where \(B_{n} = \mathrm{im} (d_{n+1} \colon C_{n+1} \to C_{n})\) is
the image of the boundary map as usual.

\begin{thm}
Suppose that \(C\) is a positive chain complex of finitely generated
\(R\)-modules with \(\mathbf{B}C\) of controlled growth such that
\((1-x)^2 \cdot \bc(\mathbf{B}C) \to 0\) as \(x \to 1^{-}\). Then
\(\mathbf{H}C\) is an Abel sequence (\emph{i.e.}, \(C\) is an Abel complex)
if and only if \(\mathbf{R}C\) is an Abel sequence, in which case
\(\ca(C) = \Alim \mathbf{H}C = \Alim \mathbf{R}C\).
\end{thm}

\begin{proof}
The short exact sequence \(0 \to Z_n C \to C_n \to B_{n-1}C \to 0\)
gives the relation \(\mathrm{rank}\, Z_n C + \mathrm{rank}\, B_{n-1} C
= \mathrm{rank}\, C_n\). Similarly, the short exact sequence
\(0 \to B_n C \to Z_n C \to H_n \to 0\) gives the relation
\(\mathrm{rank}\, Z_n C = \mathrm{rank}\, B_n C + \mathrm{rank}\, H_n
C\). Taken together, these yield
\begin{equation*}
  \mathrm{rank}\, B_n C + \mathrm{rank}\,  B_{n-1} C + \mathrm{rank}\, H_n C
  = \mathrm{rank}\, C_n
\end{equation*}
(where \(B_{-1}C\) is the trivial module), and by multiplying with
\((-1)^n x^n\) and summing up, we obtain
\begin{equation*}
  (1-x) \cdot \bc(\mathbf{B}C) + \bc(\mathbf{H}C) = \bc(\mathbf{R}C) \ .
\end{equation*}
We multiply with \((1-x)\) to arrive at
\begin{equation*}
  (1-x)^2 \cdot \bc(\mathbf{B}C) + (1-x) \cdot \bc(\mathbf{H}C)
  = (1-x) \cdot \bc(\mathbf{R}C) \ ;
\end{equation*}
in view of our hypothesis, the first term on the left vanishes in the
limit \(x \to 1^{-}\), so \(\Alim \mathbf{H}C\) exists if and only if
\(\Alim \mathbf{R}C\) exists, in which case the A-limits agree.
\end{proof}

\section{The map \(K_{0}(R) \to K^{\text{A}}_{0}(R)\)}
\label{sec:orgcc78cbb}
The "usual" algebraic \(K\)-theory of the ring~\(R\) is obtained as
the algebraic \(K\)-theory of the category~\(\fP(R)\) of (positive)
bounded chain complexes of finitely generated projective
\(R\)-modules; weak equivalences are the quasi-isomorphisms, and
cofibrations are injective chain maps.

Given a \emph{bounded} chain complex~\(C\) with finitely generated homology
modules we define the (unbounded) complex \(\Xi C = \bigoplus_{k \geq
0} C[2k]\). Apart from finitely many terms in low degrees, the
homology of the chain complex \(\Xi C\) has a very simple description:
Provided \(n\) is so large that the \(C\) is concentrated in degrees
below \(n/2\),
\begin{equation*}
  H_{n} (\Xi C) =
  \begin{cases}
    \bigoplus_{k \geq 0} H_{2k}   C & \text{if \(n\) is even,} \\
    \bigoplus_{k \geq 0} H_{2k+1} C & \text{if \(n\) is odd,}
  \end{cases}
\end{equation*}
the direct sums having finitely many non-zero terms only. This means
that, with the exception of finitely many terms of low degree, the
sequence \(\mathbf{H}\Xi C\) alternates between two values the sum of
which is~\(\chi(C)\). It follows that \(\ca(\Xi C) = \chi(C)/2\), see
Lemma~\ref{org44904eb} and Example~\ref{org2ff51a2}. It follows also that
the sequence \(|\mathbf{H}\Xi C|\) has a finite Abel limit (\emph{viz.},
half of the sum of the ranks of the homology modules of~\(C\)) so that
\(\mathbf{H}C\) is of controlled absolute growth.

The assignment \(C \mapsto \Xi C\) is functorial, preserves
quasi-isomorphisms, and maps injective chain maps to injective
maps. Thus we obtain an exact functor
\begin{equation*}
  \Xi \colon \fP(R) \to \CA \ , \quad
  C \mapsto \bigoplus_{k \geq 0} C[2k]
\end{equation*}
with values in the category~\(\CA\). The functor induces group
homomorphisms
\begin{equation*}
  \xi_{q} \colon K_{q}(R) \to K^{\text{A}}_{q}(R) \ .
\end{equation*}

Since \(R\) is a commutative ring it has the invariant basis number
property. Consequently, \(K_{0}(R)\) contains a canonical infinite
cyclic subgroup generated by the element~\([R]\), represented by~\(R\)
considered as a chain complex concentrated in degree~0. In other
words, we have an injective map \(\iota \colon \bZ \to K_{0}(R)\)
determined by \(1 \mapsto [R]\). Now the chain complex~\(\Xi R\) has
chain modules \(R\) in all even degrees, and~0 in all odd degrees, so
that \(\ca \Xi R = 1/2\). It follows that the composition
\begin{equation*}
  \bZ \xrightarrow{\iota} K_{0}(R)
      \xrightarrow{\xi_{0}} K^{\text{A}}_{0}(R)
      \xrightarrow{\ca} \bR
\end{equation*}
is the map \(t \mapsto t/2\), which is injective. Consequently, the
composition \(\bZ \xrightarrow{\iota} K_{0}(R) \xrightarrow{\xi_{0}}
K^{\text{A}}_{0}(R)\) must be injective. In other words:

\begin{thm}
The group \(K^{\text{A}}_{0}(R)\) contains a canonical infinite cyclic
subgroup, generated by the element represented by~\(\Xi R\). \qed
\end{thm}

If the ring~\(R\) is a commutative principal ideal domain, all
finitely generated projective modules are in fact finitely generated
free, with a well-defined rank. In this case \(\iota\) is an
isomorphism, with inverse given by the usual Euler characteristic
function \(\chi \colon C \mapsto \sum_{j \geq 0} (-1)^{j}
\mathrm{rank}\, C_{j}\). Hence we obtain the following:

\begin{thm}
\label{org10e3d52} Suppose that \(R\) is a commutative principal
ideal domain. The composite map \(\ca \circ \xi_{0} \colon K_{0}(R)
\to \bR\) equals the map \(\chi/2\) (that is, half of the classical
Euler characteristic). In particular, the map \(\xi_{0}\) is
injective. \qed
\end{thm}

\section*{ }
\label{sec:org2450e53}
\bibliographystyle{alpha}
\bibliography{/home/huette/SYNC-notes/Maths/Euler_characteristics/chi1}

\raggedright
\end{document}